\documentclass[11pt, twoside, a4paper, reqno]{amsart}

\usepackage{amsfonts, amssymb, amscd, amsmath, amsthm}
\usepackage[utf8]{inputenc}
\usepackage{graphicx}
\usepackage{float}
\usepackage{multicol}
\allowdisplaybreaks 
\usepackage{mathtools}
\usepackage{amsmath, amssymb}
\usepackage{amscd}
\usepackage{xcolor}
\usepackage{tikz-cd} 
\usepackage{graphicx} 
\usepackage[charter]{mathdesign}
\usepackage{hyperref}
\usepackage{enumitem}
\allowdisplaybreaks 
\usepackage{mathtools}
\newtheorem{definition}{Definition}[section]
\newtheorem{theorem}[definition]{Theorem}

\newtheorem{proposition}[definition]{Proposition}
\theoremstyle{definition}
\newtheorem{remark}[definition]{Remark}

\newtheorem{example}[definition]{Example}
\newcommand{\la}{\left\langle}
\newcommand{\ra}{\right\rangle}

\newcommand{\ucpa}{\text{UCP} \big(\mathcal{A}, \mathcal{B}  (\mathcal{H}) \big)}

\newcommand{\ucpapg}{\text{UCP}^{G_\tau} \big(\mathcal{A}, \mathcal{B}  (\mathcal{H} ) \big)}
\newcommand{\ucpapgc}{\text{UCP}^{G_\tau} \big(\mathcal{A}, \mathbb{C} \big)}

\newcommand{\cpa}{\text{CP} \big(\mathcal{A}, \mathcal{B}  (\mathcal{H} ) \big)}

\newcommand{\cpapg}{\text{CP}^{G_\tau} \big(\mathcal{A}, \mathcal{B} (\mathcal{H}) \big)}

\newcommand{\pacom}{\pi \big ( \mathcal{A} \big)^\prime}

\title[$C^\ast$-extreme points of Unital Completely Positive  maps invariant under group action]{$C^\ast$-extreme points of Unital Completely Positive  maps invariant under group action}

\author[Kulkarni]{Chaitanya J. Kulkarni}
\address{Chaitanya J. Kulkarni, Theoretical Statistics and Mathematics Unit, Indian Statistical Institute, Delhi Centre, 7 S. J.
S. Sansanwal Marg, New Delhi 110016, India}
\email{chaitanyakulkarni58@gmail.com}

\begin{document}

\maketitle

%%%%%%%%%%%%%%%%%%%%%%%%%%%%%%%%%%%%%%%%%%%%%%%%%%%%%%%%%%%%%%%%%%%%%%%%%%%%%%%%%%%%%%%
\begin{abstract}
In this work, we study a sub-collection of unital completely positive maps from a unital $C^\ast$-algebra $\mathcal{A}$ to $\mathcal{B}(\mathcal{H})$, the algebra of bounded linear operators on a Hilbert space $\mathcal{H}$ in the setting of $C^\ast$-convexity. Let $\tau$ be an action of a group $G$ on the $C^\ast$-algebra $\mathcal{A}$ through $C^\ast$-automorphisms. We focus our attention to the set of all unital completely positive maps from  $\mathcal{A}$ to $\mathcal{B}(\mathcal{H})$, which remain invariant under $\tau$. We denote this collection by the notation $\ucpapg$. This collection forms a $C^\ast$-convex set. We characterize the set of $C^\ast$-extreme points of $\ucpapg$. Further, we conclude the article by proving the Krein--Milman type theorem in the setting of $C^\ast$-convexity for the set $\ucpapg$. 
\end{abstract}

%%%%%%%%%%%%%%%%%%%%%%%%%%%%%%%%%%%%%%%%%%%%%%%%%%%%%%%%%%%%%%%%%%%%%%%%%%%%%%%%%%%%%%%

\section{Introduction} \label{sec;Introduction}
Let $\tau$ be an action of a group $G$ on a unital $C^\ast$-algebra $\mathcal{A}$ through (unital) $C^\ast$-automorphisms. Consider the following two sets:
\begin{align*}
\text{UCP} \big(\mathcal{A}, \mathbb{C} \big) &:= \left \{ \omega : \mathcal{A} \rightarrow \mathbb{C} \; \; : \; \; \omega \; \; \text{is a state} \right \}; \\
\text{UCP}^{G_\tau} \big(\mathcal{A}, \mathbb{C} \big) &:= \left \{ \omega \in  \text{UCP} \big(\mathcal{A}, \mathbb{C} \big)  \; \; : \; \; \omega(\tau_g(a)) = \omega (a) \; \; \text{for all} \; \; g \in G \; \; \text{and} \; \; a \in \mathcal{A} \right \}.
\end{align*}
The set $\text{UCP}^{G_\tau} \big(\mathcal{A}, \mathbb{C} \big )$ is a (linear) convex and compact (in weak*-topology) subset of the dual space $\mathcal{A}^\ast$.  The extreme points of $\text{UCP}^{G_\tau} \big(\mathcal{A}, \mathbb{C} \big )$ are called ergodic states. The decomposition of elements of $\text{UCP}^{G_\tau} \big(\mathcal{A}, \mathbb{C} \big)$ with respect to ergodic states is studied in \cite[Chapter 4.3]{OB1}. Further, this study has been continued with a specialization for group $G$ and its action $\tau$ \big (refer \cite[Chapter 4.3]{OB1} \big ). Next, in article \cite{BK3}, authors extended this study in the non-commutative setting. For a fixed Hilbert space $\mathcal{H}$, by following the above notations, the authors considered the following two sets.
\begin{align*}
\ucpa &:= \left \{ \phi : \mathcal{A} \rightarrow \mathcal{B}(\mathcal{H}) \; \; : \; \; \phi \; \; \text{is unital completely positive map} \right \}; \\
\ucpapg &:= \left \{ \phi \in \ucpa \; \; : \; \;  \phi(\tau_g(a)) = \phi (a) \; \; \text{for all} \; \; g \in G \; \; \text{and} \; \; a \in \mathcal{A} \right \},
\end{align*}
where $\mathcal{B}(\mathcal{H})$ denotes the algebra of all bounded linear operators on the Hilbert space $\mathcal{H}$. The set $\ucpapg$ is (linear) convex and compact (in BW-topology) subset of $\ucpa$. The characterization of the extreme points of $\ucpapg$, and correspondingly, the decomposition of elements of $\ucpapg$ with respect to its extreme points have been analyzed in \cite{BK3}. Further, this study has been done for a partial group action $\tau$ in \cite{HK}. In this article, we examine the set $\ucpapg$ using the notion of $C^\ast$-convexity which is the quantum analogue of linear convexity. 

The notion of linear convexity has been generalized to various other non-commutative versions. A few examples of such versions are as follows: The notion of $C^\ast$-convexity was introduced in \cite{LP}. Further, matrix convexity was examined in \cite{EW, O}.  In \cite{DK}, the authors introduced the concept of nc-convexity. The notion of CP-convexity was introduced in 1993 in \cite{F}. More recently, in \cite{AS}, the authors introduced the concept of $P$-$C^\ast$-convexity, which is defined corresponding to a positive operator $P$ on a Hilbert space. In \cite{Arveson1}, W. Arveson characterized the extreme points of various subsets of completely positive maps in the framework of classical (linear) convexity. In the context of $C^\ast$-convexity, the characterization of $C^\ast$-extreme points, and in particular the Krein–Milman type theorem for the various $C^\ast$-convex subsets of $\ucpa$ have been studied in  \cite{BBK, BhKu, FM, FZ, G, Z}. In \cite{FM93, W}, authors have studied the notion of $C^\ast$-convexity in some $C^\ast$-convex sets. The $C^\ast$-convex structure of the space of entanglement breaking maps on matrix algebras and its $C^\ast$-extreme points are studied in \cite{BDMS}, and in \cite{BH} this study has been done on the  space of entanglement breaking maps on operator systems. 

Throughout this article, unless stated otherwise, we adopt the following setup. Let $G$ be a group, $\mathcal{A}$ be a unital $C^\ast$-algebra, and $\mathcal{H}$ be a Hilbert space. Suppose $\tau$ is an action of $G$ on $\mathcal{A}$ through (unital) $C^\ast$-automorphisms. Let $\{ \phi_i \}^n_{i =1}$ be a collection of elements in $\ucpapg$. Then a $C^\ast$-convex combination of $\{ \phi_i \}^n_{i =1}$ is given by 
\begin{equation*}
\phi(\cdot) := \sum^n_{i =1} T^\ast_i \phi_i(\cdot) T_i, \; \; \; \text{where each} \; \; T_i \in \mathcal{B}(\mathcal{H}) \; \; \; \text{and} \; \; \sum\limits^n_{i =1} T^\ast_i T_i = \mathrm{Id}_\mathcal{H}.
\end{equation*}
Then $\phi \in \ucpapg$. This shows that $\ucpapg$ is a $C^\ast$-convex subset of the set $\ucpa$. Moreover, $\ucpapg$ is a compact subset of $\ucpa$ with respect to BW-topology. This naturally leads to the study of $C^\ast$-extreme points of $\ucpapg$ (see Section \ref{sec;C star extreme points Sufficient Conditions}). In this article, we first provide some sufficient conditions for an element of $\ucpapg$ to be a $C^\ast$-extreme point. We then characterize the set of $C^\ast$-extreme points of $\ucpapg$. We refer to \cite{BK3}, for the notion of the Radon-Nikodym type theorem in $\ucpapg$. We also use the techniques developed in \cite{BDMS, FM, FZ}. This article also discusses a Krein–Milman type theorem for the set $\ucpapg$ in the framework of $C^\ast$-convexity. The result is obtained by adapting the Krein-–Milman type theorems established in \cite{BBK, BhKu, FM}.

We organize this article as follows. In Section \ref{sec;Preliminaries}, we briefly recall the results and notions required for the remainder of the article. In Section \ref{sec;C star extreme points Sufficient Conditions}, we introduce the concept of $C^\ast$-extreme points in the set $\ucpapg$. We then discuss the relationship between $C^\ast$-extreme points and linear extreme points of $\ucpapg$. Further, we present some sufficient conditions that ensure a given element of $\ucpapg$ is a $C^\ast$-extreme point. In Section~\ref{sec; characterization}, we characterize the set of all $C^\ast$-extreme points of $\ucpapg$. Finally, in Section~\ref{sec; KM Theorem}, we prove a Krein–-Milman type theorem for the set $\ucpapg$ in the setting of $C^\ast$-convexity.

%%%%%%%%%%%%%%%%%%%%%%%%%%%%%%%%%%%%%%%%%%%%%%%%%%%%%%%%%%%%%%%%%%%%%%%%%%%%%%%%%%%%%%%

\section{Preliminaries} \label{sec;Preliminaries}

Here we recall some important results that will be used in the remainder of the article. We refer mainly to \cite{Arveson1, BK3, Paulsen} for the results mentioned in this section. 

%%%%%%%%%%%%%%%%%%%%%%%%%%%%%%%%%%%%%%%%%%%%%%%%%%%%%%%%%%%%%%%%%%%%%%%%%%%%%%%%%%%%%%%

\subsection{Stinespring's dilation of a completely positive map}
First, we recall the Stinespring’s dilation and then the Radon--Nikodym derivative for completely positive maps.  

\begin{theorem}[{\cite[Theorem 4.1]{Paulsen}}]\label{thm;sd}
Let $\mathcal{A}$ be a unital $C^\ast$-algebra and $\mathcal{H}$ be a Hilbert space. Let $\phi : \mathcal{A} \rightarrow \mathcal{B}(\mathcal{H})$ be a completely positive map. Then there exists a Hilbert space $\mathcal{K}$, a bounded linear map $V : \mathcal{H} \rightarrow \mathcal{K}$, and a unital \(^*\)-homomorphism $\pi : \mathcal{A} \rightarrow \mathcal{B}(\mathcal{K})$ such that
\begin{equation*}
\phi(a) = V^*\pi(a)V, \; \; \; \;  \text{for all $a \in \mathcal{A}$}. 
\end{equation*} 
Moreover, the set $\big \{ \pi(a)Vh : a \in \mathcal{A}, \; h \in \mathcal{H} \big \}$ spans a dense subspace of $\mathcal{K}$, and if $\phi$ is unital, then $V : \mathcal{H} \rightarrow \mathcal{K}$ is an isometry.
\end{theorem}

Corresponding to a completely positive map $\phi : \mathcal{A} \rightarrow \mathcal{B}(\mathcal{H})$, a triple $(\pi, V, \mathcal{K} )$ obtained in Theorem~\ref{thm;sd} is called \textit{a minimal Stinespring representation} for $\phi$. The following proposition states that for a given completely positive map $\phi$, a minimal Stinespring representation is unique up to a unitary equivalence.  

\begin{proposition}[{\cite[Proposition 4.2]{Paulsen}}] \label{prop;umsd}
Let $\mathcal{A}$ be a unital $C^\ast$-algebra, $\mathcal{H}$ be a Hilbert space and let $\phi : \mathcal{A} \rightarrow \mathcal{B}(\mathcal{H})$ be a completely positive map. Suppose $\big (\pi_i, V_i, \mathcal{K}_i \big)_{i = 1, 2}$ be two minimal Stinespring representations for $\phi$. Then there exits a unitary $U : \mathcal{K}_1 \rightarrow \mathcal{K}_2$ such that $UV_1 = V_2$ and $U \pi_1 U^* = \pi_2$. 
\end{proposition}

\noindent 
Let  $\phi : \mathcal{A} \rightarrow \mathcal{B}(\mathcal{H})$ be a completely positive map, and let $\big (\pi, V, \mathcal{K} \big )$ be its minimal Stinespring representation. By following Proposition \ref{prop;umsd}, we refer $\big (\pi, V, \mathcal{K} \big )$ as \textit{the minimal Stinespring representation} for $\phi$. Next, for a unital $C^\ast$-algebra $\mathcal{A}$ and a Hilbert space $\mathcal{H}$, we consider the set
\begin{equation*}
\cpa := \big \{ \phi : \mathcal{A} \rightarrow \mathcal{B}(\mathcal{H}) \; \; : \; \;  \phi \; \; \text{is a completely positive map} \big \},
\end{equation*} 
with a partial order "$\leq$". For $\phi_1, \phi_2 \in \cpa$ 
\begin{equation*}
\phi_1 \leq \phi_2, \; \; \; \text{whenever,} \; \; \; \phi_2 - \phi_1 \in \cpa.
\end{equation*}
For each fixed $\phi \in \cpa$, consider the set
\begin{equation*}
[0, \phi] = \big \{ \theta \in \cpa \; \; : \; \; \theta \leq \phi \big \}.
\end{equation*} 
Let $\phi \in \cpa$ and $\big (\pi, V, \mathcal{K} \big )$ be the minimal Stinespring representation for $\phi$. Then for each fixed $T \in \pacom$ define a map 
$\phi_T : \mathcal{A} \rightarrow \mathcal{B}(\mathcal{H})$ by $\phi_T(\cdot) := V^* \pi (\cdot) T V$. From \cite{Arveson1}, now we recall the Radon--Nikodym type result for completely positive maps 

\begin{theorem} [{\cite[Theorem 1.4.2]{Arveson1}}] \label{thm;arnd}
Let $\phi \in \cpa$ and $\big (\pi, V, \mathcal{K} \big )$ be the minimal Stinespring representation for $\phi$. Then there exists an affine order isomorphism from the partially ordered convex set of operators $ \big \{T \in \pacom \; \; : \; \;  0 \leq T \leq  1_\mathcal{K} \}$ onto $[0, \phi]$, which is given by the map 
\begin{equation*}
T \mapsto \phi_T(\cdot) := V^* \pi (\cdot) T V.
\end{equation*}
\end{theorem}

Next, we briefly recall some important results from \cite{BK3}. For a unital $C^\ast$-algebra $\mathcal{A}$ and a Hilbert space $\mathcal{H}$, we recall the following notations:
\begin{align*}
%\cba &= \big \{ \phi : \mathcal{A} \rightarrow \mathcal{B}(\mathcal{H}) \; \; : \; \;  \phi \; \; \text{is a completely bounded map} \big \}; \\
\cpa &= \big \{ \phi : \mathcal{A} \rightarrow \mathcal{B}(\mathcal{H}) \; \; : \; \;  \phi \; \; \text{is a completely positive map} \big \}; \\
\ucpa &= \big \{ \phi : \mathcal{A} \rightarrow \mathcal{B}(\mathcal{H}) \; \; : \; \;  \phi \; \; \text{is a unital completely positive map} \big \}.
\end{align*}
Let $G$ be a group and $\tau$ be an action of $G$ on $\mathcal{A}$ through (unital) $C^\ast$-automorphisms. Then consider the following sets:
\begin{align*}
\cpapg &= \big \{ \phi \in \cpa  \; \; : \; \;   \phi (\tau_g(a)) = \phi(a), \; \; \text{for all} \; g \in G \; \; \text{and} \; \; a \in \mathcal{A} \big \}; \\
\ucpapg &:= \left \{ \phi \in \ucpa \; \; : \; \;  \phi(\tau_g(a)) = \phi (a), \; \; \text{for all} \; \; g \in G \; \; \text{and} \; \; a \in \mathcal{A} \right \}.
\end{align*}
Let $\phi \in \ucpapg$, and let the minimal Stinespring representation of $\phi$ be given by $\big (\pi, V, \mathcal{K} \big )$. Since $\phi \in \ucpapg$, for $a\in \mathcal{A}$ and $g \in G$, we have 
\begin{equation*}
\phi(a) = V^* \pi(a) V = \phi(\tau_g(a)) = V^* \pi(\tau_g(a)) V.
\end{equation*} 
This implies that for a fixed $g \in G$ the minimal Stinespring dilation of $\phi$ can also be given by $\big (\pi \circ \tau_g, V, \mathcal{K} \big )$. Now, by the uniqueness of the minimal Stinespring dilation (see Proposition \ref{prop;umsd}), we get a unitary operator, say $U_g : \mathcal{K} \rightarrow \mathcal{K}$ satisfying 
\begin{equation*}
U_gV = V  ~~~ \text{and} ~~~ U_g \pi U^*_g = \pi \circ \tau_g.
\end{equation*}
Since $g \in G$ was arbitrarily chosen, we get a representation of the group $G$ on the Hilbert space $\mathcal{K}$ which we denote by $U_\phi : G \rightarrow \mathcal{B}(\mathcal{K})$, and is defined as:
\begin{equation*}
U_\phi(g) := U_g, \; \; \; \; \;  \text{for all} \; \; \; g \in G.
\end{equation*}
Moreover, we have
\begin{equation*}
U_\phi(g)V = V  \; \; \; \text{and} \; \; \; U_\phi(g) \pi(a) U_\phi(g)^* = \pi \circ \tau_g(a), \; \; \; \; \;  \text{for all} \; \; \; a \in A, ~ g \in G.
\end{equation*}

\begin{remark} \label{rem; unitary rep of G}
Let $\phi \in \ucpapg$, and let the minimal Stinespring representation of $\phi$ be given by $\big (\pi, V, \mathcal{K} \big )$. For a fixed $\phi \in \ucpapg$, from the discussion above, we get a representation of the group $G$ on the Hilbert space $\mathcal{K}$ denoted by $U_\phi : G \rightarrow \mathcal{B}(\mathcal{K})$, and which satisfies
\begin{equation*}
U_\phi(g)V = V  \; \; \; \text{and} \; \; \; U_\phi(g) \pi(a) U_\phi(g)^* = \pi \circ \tau_g(a), \; \; \; \; \;  \text{for all} \; \; \; a \in A, ~ g \in G.
\end{equation*}
\end{remark}

The following theorem established in \cite{BK3} gives the Radon--Nikodym type result in the space $\cpapg$. 

\begin{theorem} [{\cite[Lemma 4.1]{BK3}}] \label{lem;rnd}
Let $\phi \in \ucpapg$, and let the minimal Stinespring representation of $\phi$ be given by $\big (\pi, V, \mathcal{K} \big )$. Then, the map $T \mapsto \phi_T$, where
\begin{equation*}
\phi_T(a) := V^\ast \pi(a) T V \; \; \; \text{for all} \; \; a \in \mathcal{A},
\end{equation*}
is an affine isomorphism from the partially ordered convex set of operators $\big \{ T \in  \big (\pi(\mathcal{A}) \cup U_\phi(G) \big )^\prime \; \; : \; \;  0 \leq T \leq \mathrm{Id}_\mathcal{K} \big \}$ onto $[0, \phi] \cap \cpapg$.
\end{theorem}

%%%%%%%%%%%%%%%%%%%%%%%%%%%%%%%%%%%%%%%%%%%%%%%%%%%%%%%%%%%%%%%%%%%%%%%%%%%%%%%%%%%%%%%%%

\section{$C^\ast$-extreme Points: Some Sufficient Conditions} \label{sec;C star extreme points Sufficient Conditions}

Let $G$ be a group, $\mathcal{A}$ be a unital $C^\ast$-algebra, and $\mathcal{H}$ be a Hilbert space. Suppose $\tau$ is an action of $G$ on $\mathcal{A}$ through (unital) $C^\ast$-automorphisms. We quickly recall the following notations: 
\begin{align*}
\ucpa &:= \left \{ \phi : \mathcal{A} \rightarrow \mathcal{B}(\mathcal{H}) \; \; : \; \; \phi \; \; \text{is unital completely positive map} \right \}; \\
\ucpapg &:= \left \{ \phi \in \ucpa \; \; : \; \;  \phi(\tau_g(a)) = \phi (a), \; \; \text{for all} \; \; g \in G \; \; \text{and} \; \; a \in \mathcal{A} \right \}.
\end{align*}

\begin{example} \label{ex;Ginv} We now present a couple of examples of maps in the set $\ucpapg$.
\begin{enumerate}
\item Let $G$ be a finite group, $\mathcal{A}$ be a unital $C^\ast$-algebra, and $\mathcal{H}$ be a Hilbert space. Suppose $\tau$ is an action of $G$ on $\mathcal{A}$ through (unital) $C^\ast$-automorphisms. Let $\phi \in \ucpa$. Define a map $\tilde{\phi} : \mathcal{A} \rightarrow \mathcal{B}(\mathcal{H})$ by
\begin{equation*}
\tilde{\phi}(a) := \phi \left ( \frac{\sum\limits_{g \in G} \tau_g(a)}{|G|} \right ), \;  \; \; \; \text{for all} \; \; a \in \mathcal{A}.
\end{equation*}
Then $\tilde{\phi} \in \ucpapg$

\item Let $G$ be a compact group, and let $\lambda$ be the Haar measure on $G$. Fix a Hilbert space $\mathcal{H}$, and a $C^\ast$-subalgebra $\mathcal{A}$ of $\mathcal{B}(\mathcal{H})$. Let $\{ W_g \}_{g \in G}$ be a family of unitaries in $\mathcal{A}$, and $\tau$ be an action of the group $G$ on $\mathcal{A}$ that is given by 
\begin{equation*}
\tau_g(a) := W_g^* a W_g, \; \; \; \text{for all} \; \; a \in \mathcal{A}.
\end{equation*}
Consider a map $\phi : \mathcal{A} \rightarrow \mathcal{B}(\mathcal{H})$ defined by
\begin{equation*}
\phi(a) := \int_{G} W_g^*aW_g \, \mathrm{d} \lambda, \; \; \; \text{for all} \; \; a \in \mathcal{A}.
\end{equation*}
where the integral is observed in a weak sense. That is, for $h_1, h_2 \in \mathcal{H}$, we have
\begin{equation*}
\big \langle \phi(a)h_1, h_2 \big \rangle =  \int_{G} \la W_g^*aW_gh_1, h_2 \ra \, \mathrm{d} \lambda.
\end{equation*}
Then $\phi$ is a unital completely positive map. Now fix some $g_0 \in G$. Then for any $a \in \mathcal{A}$ and $h_1, h_2 \in \mathcal{H}$, we have
\begin{eqnarray*}
\la \phi(\tau_{g_0}(a))h_1, h_2 \ra &=&  \int_{G} \la W_g^*\tau_{g_0}(a)W_gh_1, h_2 \ra \, \mathrm{d} \lambda \\
&=& \int_{G} \la W_g^*W_{g_0}^* a W_{g_0}W_gh_1, h_2 \ra \, \mathrm{d} \lambda \\
&=& \int_{G} \la W_{g_0 g}^* a W_{g_0g}h_1, h_2 \ra \, \mathrm{d} \lambda \\
&=& \la \phi(a)h_1, h_2 \ra.
\end{eqnarray*} 
This shows that the map $\phi \in \ucpapg$.
\end{enumerate}
\end{example}

Let $G$ be a group, $\mathcal{A}$ be a unital $C^\ast$-algebra, and $\mathcal{H}$ be a Hilbert space. Suppose $\tau$ is an action of $G$ on $\mathcal{A}$ through (unital) $C^\ast$-automorphisms. Let $\{ \phi_i \}^n_{i =1}$ be a collection of elements in $\ucpapg$. Then a $C^\ast$-convex combination of $\{ \phi_i \}^n_{i =1}$ is given by 
\begin{equation*}
\phi(\cdot) := \sum^n_{i =1} T^\ast_i \phi_i(\cdot) T_i, \; \; \; \text{where each} \; \; T_i \in \mathcal{B}(\mathcal{H}) \; \; \; \text{and} \; \; \sum\limits^n_{i =1} T^\ast_i T_i = \mathrm{Id}_\mathcal{H}.
\end{equation*}
One may see that $\phi$ is a unital completely positive map, and moreover
\begin{equation*}
\phi(\tau_g(a)) = \sum^n_{i =1} T^\ast_i \phi_i(\tau_g(a)) T_i =  \sum^n_{i =1} T^\ast_i \phi_i(a) T_i = \phi (a)
\end{equation*}
for all $g \in G$ and $a \in \mathcal{A}$. Thus $\phi \in \ucpapg$. This shows that $\ucpapg$ is a $C^\ast$-convex subset of $\ucpa$. Also, the set $\ucpapg$ is a closed subset of $\ucpa$ and hence compact in BW-topology (see \cite{BK3}). Following the notion of $C^\ast$-extreme point in $\ucpa$ introduced in \cite{FM}, we call an element $\phi \in \ucpapg$ to be a $C^\ast$-extreme point of $\ucpapg$ if, whenever
$\phi$ is written as a \textit{proper} $C^\ast$-convex combination, that is, 
\begin{equation*}
\phi(\cdot) := \sum^n_{i =1} T^\ast_i \phi_i(\cdot) T_i, \; \; \text{where each} \; \; \phi_i \in \ucpapg, \; \; T_i \in \mathcal{B}(\mathcal{H}), \; \; \; \text{and} \; \; \sum\limits^n_{i =1} T^\ast_i T_i = \mathrm{Id}_\mathcal{H}.
\end{equation*}
with $T^{-1}_i$ exists for all $1 \leq i \leq n$, then there exist unitaries $\big \{ U_i \in  \mathcal{B}(\mathcal{H}) \big \}^n_{i =1}$ such that
\begin{equation*}
\phi_i(\cdot) = U^\ast_i \phi(\cdot) U_i, \; \; \; \text{for all} \; \; 1 \leq i \leq n.
\end{equation*}
In such case, we denote this by $\phi_i \sim \phi$ for all $1 \leq i \leq n$.

In this section, we present some sufficient conditions under which a completely positive map in $\ucpapg$ is a $C^\ast$-extreme point. Before this, when $\mathcal{H}$ is finite dimensional, we show that every $C^\ast$-extreme point of $\ucpapg$ is also a linear extreme point of $\ucpapg$.

\begin{proposition} \label{prop; C star extreme implies extreme}
Let $\mathcal{H}$ be a finite dimensional Hilbert space, and let $\phi$ be a $C^\ast$-extreme point of $\ucpapg$. Then $\phi$ is an extreme point of $\ucpapg$.
\end{proposition}
\begin{proof}
Let $\phi(\cdot) = \lambda \phi_1(\cdot) + (1 - \lambda) \phi_2(\cdot)$, where $\phi_1, \phi_2 \in \ucpapg$ and $\lambda \in (0, 1)$. This implies $\phi(\cdot) = (\lambda)^\frac{1}{2} \mathrm{Id}_{\mathcal{H}} \phi_1(\cdot) (\lambda)^\frac{1}{2} \mathrm{Id}_{\mathcal{H}} + (1 - \lambda)^\frac{1}{2} \mathrm{Id}_{\mathcal{H}}  \phi_2(\cdot) (1 - \lambda)^\frac{1}{2} \mathrm{Id}_{\mathcal{H}}$. Since $\phi$ is a $C^\ast$-extreme point of $\ucpapg$, we get $\phi \sim \phi_1 \sim \phi_2$. Thus there exist $U_1, U_2$ unitaries in $\mathcal{B}(\mathcal{H})$ such that for each $a \in \mathcal{A}$, we have  
\begin{equation} \label{eq; phi a is unitarily equi to phi i of a}
\phi(a) = U_i^\ast \phi_i(a) U_i \; \; \; \; \text{for} \; \; i \in \{ 1, 2 \}.
\end{equation}
Now by recalling that $\mathcal{H}$ is a finite dimensional Hilbert space, we get from \cite{L} that $\phi(a)$ is an extreme point of the convex hull of its unitary orbit for every $a \in \mathcal{A}$. Thus, $\phi(a) = \phi_1(a) = \phi_2(a)$ for each $a \in \mathcal{A}$. Hence, we get the result.
\end{proof}

In the remaining part of this section, we prove some results that give sufficient conditions for an element $\phi \in \ucpapg$ to be a $C^\ast$-extreme point of $\ucpapg$.

\begin{definition}
Let $\omega : \mathcal{A} \rightarrow \mathbb{C}$ be a state. Then an inflation $\phi_\omega \in \ucpa$ corresponding to $\omega$ is given by $\phi_\omega(a) : = \omega(a) \mathrm{Id}_{\mathcal{H}}$. 
\end{definition}

\begin{proposition} \label{prop; sufficient conditions}
Consider $\phi \in \ucpapg$ with the minimal Stinespring representation for $\phi$ is given by $\big (\pi, V, \mathcal{K} \big )$. If one of the following conditions holds, then $\phi$ is a $C^\ast$-extreme point of $\ucpapg$.
\begin{enumerate}
\item $\phi$ is an inflation of some pure state $\omega \in \ucpapgc$;
\item $\phi$ is multiplicative;
\item $V\mathcal{H}$ is invariant under $\big (\pi(\mathcal{A}) \cup U_\phi(G) \big )^\prime$, where $U_\phi$ is the representation of $G$ as given in Remark \ref{rem; unitary rep of G};
\item $\phi$ is a pure element of $\cpa$.
\end{enumerate}
\end{proposition}
\begin{proof}
\textbf{Proof of (1):} Suppose $\phi$ is an inflation of a pure state $\omega \in \ucpapgc$. Let $\phi(\cdot) = \sum\limits^k_{i =1} T_i^\ast \phi_i(\cdot) T_i$ be a proper $C^\ast$-convex combination, where each $\phi_i \in \ucpapg$. Fix a unit vector $\xi \in \mathcal{H}$. For each $i \in \{ 1, 2, \cdots, k \}$, define a map $\omega^\xi_i \colon \mathcal{A} \rightarrow \mathbb{C}$ by 
\begin{equation} \label{eqn; omega xi i}
\omega^\xi_i(a) := \frac{1}{\| T_i \xi \|^2} \big \langle \phi_i(a)T_i \xi, T_i \xi \big  \rangle, \; \; \; \text{for} \; \; a \in \mathcal{A}.
\end{equation}
Then by following Equation \eqref{eqn; omega xi i}, we get
\begin{equation*}
\omega(a) = \big \langle \omega(a) \mathrm{Id}_{\mathcal{H}} \xi, \xi \big \rangle = \big \langle \phi(a) \xi, \xi \big \rangle = \sum\limits^k_{i =1} \big \langle \phi_i(a) T_i\xi, T_i\xi \big \rangle = \sum\limits^k_{i =1} \big \| T_i \xi \big \|^2 \omega^\xi_i(a).
\end{equation*}
As $\omega$ is a pure state, we obtain $\omega = \omega^\xi_i$, for each $i \in \{ 1, 2, \cdots, k \}$. We note that the unit vector $\xi \in \mathcal{H}$ was chosen arbitrarily. This observation will be used in the following computations. Fix $i \in \{ 1, 2, \cdots, k \}$, and a unit vector $h \in \mathcal{H}$. Since $T_i$ is invertible, there exists a unit vector $h^\prime \in \mathcal{H}$ such that $\frac{T_ih^\prime}{\|T_ih^\prime \|} = h$. Thus, for each $a \in \mathcal{A}$, we get
\begin{equation*}
\big \langle \phi(a)h, h \big \rangle = \omega(a)  = \omega^{h^\prime}_i(a) = \left \langle \phi_i(a)\frac{T_ih^\prime}{\|T_ih^\prime \|} , \frac{T_ih^\prime}{\|T_ih^\prime \|}  \right \rangle =  \big \langle \phi_i(a)h, h \big \rangle.
\end{equation*}
As the unit vector $h \in \mathcal{H}$, and $i \in \{ 1, 2, \cdots, k \}$ were arbitrarily chosen, we obtain $\phi = \phi_i$ for each $i \in \{ 1, 2, \cdots, k \}$. This proves that  $\phi$ is a $C^\ast$-extreme point of $\ucpapg$.

\noindent
\textbf{Proof of (2):} Suppose that $\phi$ is multiplicative and $\phi(\cdot) = \sum\limits^k_{i =1} T_i^\ast \phi_i(\cdot) T_i$ is a proper $C^\ast$-convex combination, where each $\phi_i \in \ucpapg$. Since $\phi$ is multiplicative, $\phi$ preserves adjoints, and $\phi$ is unital it follows that $\phi$ acts on $\mathcal{H}$ non-degenerately. Thus the minimal Stinespring representation of $\phi$ is given by $\big (\pi = \phi, V = \mathrm{Id}_{\mathcal{H}}, \mathcal{K} = \mathcal{H} \big )$. For each fixed $i$, we have  $T_i^\ast \phi_i T_i \leq \phi$, and hence we obtain an operator $S_i \in \phi(\mathcal{A})^\prime$ such that $0 \leq S_i \leq \mathrm{Id}_{\mathcal{H}}$ and $T_i^\ast \phi_i(\cdot) T_i = S_i\phi(\cdot)$ (see Theorem \ref{thm;arnd}). This implies that $\phi_i(\cdot) = (T^\ast_i)^{-1} S_i^{\frac{1}{2}} \phi(\cdot)  S_i^{\frac{1}{2}} (T_i)^{-1}$. Now it remains to show that $S_i^{\frac{1}{2}} (T_i)^{-1} $ is a unitary operator. First observe that  $T_i^\ast T_i = T_i^\ast \phi_i(1_\mathcal{A}) T_i = S_i\phi(1_\mathcal{A}) = S_i$. Using this, we get 
\begin{equation*}
\big (  S_i^{\frac{1}{2}} (T_i)^{-1} \big )^\ast S_i^{\frac{1}{2}} (T_i)^{-1} = (T^{-1}_i)^\ast S_i^{\frac{1}{2}} S_i^{\frac{1}{2}} (T_i)^{-1} = (T^\ast_i)^{-1} S_i(T_i)^{-1} = (T^\ast_i)^{-1} T^\ast_i T_i (T_i)^{-1} = \mathrm{Id}_\mathcal{H},
\end{equation*}
and similarly, $ S_i^{\frac{1}{2}} (T_i)^{-1} \big (  S_i^{\frac{1}{2}} (T_i)^{-1} \big )^\ast = \mathrm{Id}_\mathcal{H}$. Finally, as $i$ was chosen arbitrarily, we get that $\phi$ is a $C^\ast$-extreme point of $\ucpapg$.

\textbf{Proof of (3):} Suppose $V\mathcal{H}$ is invariant under $\big (\pi(\mathcal{A}) \cup U_\phi(G) \big )^\prime$, where $\big (\pi, V, \mathcal{K} \big )$ is the minimal Stinespring representation for $\phi$ and $U_\phi$ is the unitary representation of $G$ as given in Remark \ref{rem; unitary rep of G}. Let $\phi(\cdot) = \sum\limits^k_{i =1} T_i^\ast \phi_i(\cdot) T_i$ be a proper $C^\ast$-convex combination, where each $\phi_i \in \ucpapg$. For each fixed $i$, by following Theorem \ref{lem;rnd}, there exists a unique $S_i \in \big (\pi(\mathcal{A}) \cup U_\phi(G)\big )^\prime$ such that $0 \leq S_i \leq \mathrm{Id}_{\mathcal{K}}$ and $T_i^\ast \phi_i(\cdot) T_i = V^\ast S_i \pi(\cdot) V$. Define $R_i = V^\ast S_i^{\frac{1}{2}} V$. By using the assumption, we know that the projection $VV^\ast$ commutes with $(\pi(\mathcal{A}) \cup U_\phi(G)\big )^\prime$. Thus, we get 
\begin{equation} \label{eqn; vri = sihalf v}
VR_i = VV^\ast S_i^{\frac{1}{2}} V = S_i^{\frac{1}{2}} VV^\ast V =  S_i^{\frac{1}{2}} V.
\end{equation}
Also, by observing $T_i^\ast T_i = T_i^\ast \phi_i(1_\mathcal{A}) T_i = V^\ast S_i \pi(1_\mathcal{A}) V = V^\ast S_i V$, we obtain
\begin{equation*}
R_i = (R_i^2)^\frac{1}{2} = \big (V^\ast S_i^{\frac{1}{2}} V V^\ast S_i^{\frac{1}{2}} V \big )^\frac{1}{2}  = \big (V^\ast S_i^{\frac{1}{2}} S_i^{\frac{1}{2}} V V^\ast  V \big )^\frac{1}{2} = (V^\ast S_i V)^\frac{1}{2} =  (T_i^\ast T_i)^\frac{1}{2}.
\end{equation*}
Since $T_i$ is invertible, it follows that $R_i$ is invertible. By following Equation \eqref{eqn; vri = sihalf v}, we get
\begin{equation*}
T_i^\ast \phi_i(\cdot) T_i = V^\ast S_i^\frac{1}{2} \pi(\cdot) S_i^\frac{1}{2} V = R_i^\ast V^\ast \pi(\cdot) V R_i = R^\ast_i \phi(\cdot) R_i,
\end{equation*}
and hence
\begin{equation*}
\phi_i(\cdot) = (R_i T_i^{-1})^\ast \phi(\cdot) (R_i T_i^{-1}), \; \; \; \; \;\; \phi(\cdot) = (T_i R^{-1}_i)^\ast \phi_i(\cdot) (T_i R^{-1}_i).
\end{equation*}
Now, it remains to show that $R_i T_i^{-1}$ is unitary. We have
\begin{align*}
\mathrm{Id}_\mathcal{H} &= \phi_i (1_\mathcal{A}) = (R_i T_i^{-1})^\ast \phi(1_\mathcal{A})  (R_i T_i^{-1}) = (R_i T_i^{-1})^\ast(R_i T_i^{-1}); \\
\mathrm{Id}_\mathcal{H} &= \phi (1_\mathcal{A}) = (T_i R^{-1}_i)^\ast \phi_i (1_\mathcal{A}) (T_i R^{-1}_i) = (T_i R^{-1}_i)^\ast (T_i R^{-1}_i) \implies \mathrm{Id}_\mathcal{H} = (R_i T^{-1}_i)(R_i T^{-1}_i)^\ast.
\end{align*}
This shows that $R_i T_i^{-1}$ is unitary. Since $i$ was arbitrarily chosen, we get the result.

\textbf{Proof of (4):} Suppose that $\phi$ is pure in $\cpa$. By following \cite[Corollary 1.4.3]{Arveson1}, we get that $\pi$ is an irreducible representation of $\mathcal{A}$ on $\mathcal{K}$. Thus $(\pi(\mathcal{A}))^\prime \cong \mathbb{C}$, and hence $(\pi(\mathcal{A}) \cup U_\phi(G)\big )^\prime \cong \mathbb{C}$. Thus the result follows from (3).
\end{proof}

%\begin{corollary}
%Consider $\phi \in \ucpapg$ with the minimal Stinespring representation for $\phi$ is given by $\big (\pi, V, \mathcal{K} \big )$. If one of the following conditions holds, then $\phi$ is an extreme point of $\ucpapg$.
%\begin{enumerate}
%\item $\phi$ is an inflation of some pure state $\omega \in \ucpapgc$;
%\item $\phi$ is multiplicative;
%\item $V\mathcal{H}$ is invariant under $\big (\pi(\mathcal{A}) \cup U^{G_\tau}_\phi \big )^\prime$;
%\item $\phi$ is pure.
%\end{enumerate}
%\end{corollary}
%\begin{proof}

%\end{proof}

\begin{definition} \cite{FM}
Let $\phi_i \in \text{UCP} \big(\mathcal{A}, \mathcal{B}  (\mathcal{H}_i) \big)$ be a pure map for each $1 \leq i \leq n$ with the minimal Stinespring representation for $\phi_i$ be given by $\big (\pi_i, V_i, \mathcal{K}_i \big )$. Then $\phi_1, \phi_2, \cdots, \phi_n$ are said to be disjoint, if the irreducible representations $\pi_1, \pi_2, \cdots, \pi_n$ are mutually nonequivalent.
\end{definition}

\begin{proposition}
Let $\phi_i \in \text{UCP}^{G_\tau} \big(\mathcal{A}, \mathcal{B}  (\mathcal{H}_i ) \big)$ be pure for each $1 \leq i \leq n$. If $\phi_1, \phi_2, \cdots, \phi_n$ are disjoint, then $\phi_1 \oplus \phi_2 \oplus \cdots \oplus \phi_n$ is a $C^\ast$-extreme point of $\ucpapg$, where $\mathcal{H} = \mathcal{H}_1 \oplus \mathcal{H}_2 \oplus \cdots \oplus \mathcal{H}_n$.
\end{proposition}
\begin{proof}
Let $\phi := \phi_1 \oplus \phi_2 \oplus \cdots \oplus \phi_n : \mathcal{A} \rightarrow \mathcal{B}(\mathcal{H} = \mathcal{H}_1 \oplus \mathcal{H}_2 \oplus \cdots \oplus \mathcal{H}_n)$. Then $\phi \in \ucpapg$.  Let $\big (\pi_i, V_i, \mathcal{K}_i \big )$ be the minimal Stinespring representation for $\phi_i$ and let $P_i : \mathcal{H} \rightarrow \mathcal{H}_i$ be the projection for each $i \in \{ 1, 2, \cdots, n \}$. Define 
\begin{equation*}
\mathcal{K} := \mathcal{K}_1 \oplus \mathcal{K}_2 \oplus \cdots \oplus \mathcal{K}_n, \; \; \; \; \;   V : \mathcal{H} \rightarrow \mathcal{K} \; \; \; \text{by} \; \; \; Vh := V_1P_1h \oplus V_2P_2h \oplus \cdots \oplus V_nP_nh,
\end{equation*}
fo all $h \in \mathcal{H}$, and
\begin{equation*}
\pi : \mathcal{A} \rightarrow \mathcal{B}(\mathcal{K}), \;  \; \; \; \text{by} \; \; \; \pi := \pi_1 \oplus \pi_2 \oplus \cdots \oplus \pi_n. 
\end{equation*}
Then the minimal Stinespring's representation of $\phi$ is given by $\big (\pi, V, \mathcal{K} \big )$ (see \cite[Remark 3.1]{BhKu}). Moreover, by following \cite[Lemma 3.3]{BhKu}, we obtain that $\pi(\mathcal{A})^\prime \cong \mathbb{C}^n$. Then  the result follows from (3) of Proposition \ref{prop; sufficient conditions}.
\end{proof}

\section{Characterization of $C^\ast$-extreme points} \label{sec; characterization}

In this section, we provide a characterization of the $C^\ast$-extreme points of $\ucpapg$. We present some criteria that help in identifying these $C^\ast$-extreme points. We begin by showing that, in order to verify whether a given element of $\ucpapg$ is $C^\ast$-extreme it is enough to consider the $C^\ast$-convex combinations involving only two elements of $\ucpapg$.

\begin{theorem}  \label{thm; equivalent condition 3}
Let $\phi \in \ucpapg$ with the minimal Stinespring representation for $\phi$ be given by $\big (\pi, V, \mathcal{K} \big )$. Then the following statements are equivalent
\begin{enumerate}
\item $\phi$ is a $C^\ast$-extreme point of $\ucpapg$;
\item if  $\phi(\cdot) = \sum\limits^2_{i =1} T^\ast_i \phi_i(\cdot) T_i$, where each $\phi_i \in \ucpapg$ and $T_i \in \mathcal{B}(\mathcal{H})$ with $T^{-1}_i$ exists for all $i = 1, 2$ and $\sum\limits^2_{i =1} T^\ast_i T_i = \mathrm{Id}_\mathcal{H}$, then $\phi_i$ is unitarily equivalent to $\phi$ for $i = 1, 2$.
\end{enumerate}
\end{theorem}
\begin{proof}
One may see that the proof of (1) $\implies$ (2) follows from the definition of $C^\ast$-extreme point of $\ucpapg$. Thus, it only remains to prove (2) $\implies$ (1). Now assume that $\phi(\cdot) = \sum\limits^n_{i =1} T^\ast_i \phi_i(\cdot) T_i$, where each $\phi_i \in \ucpapg$ and $T_i \in \mathcal{B}(\mathcal{H})$ with $T^{-1}_i$ exists for all $i \in \{ 1, 2, \cdots, n \}$ and $\sum\limits^n_{i =1} T^\ast_i T_i = \mathrm{Id}_\mathcal{H}$. From the assumption, we know that if $n = 2$, then each $\phi_i$ is unitarily equivalent to $\phi$. Assume that the result is true for $m \geq 2$ and suppose $\phi(\cdot) = \sum\limits^{m+1}_{i =1} T^\ast_i \phi_i(\cdot) T_i$, where each $\phi_i \in \ucpapg$ and $T_i \in \mathcal{B}(\mathcal{H})$ with $T^{-1}_i$ exists for all $i \in \{ 1, 2, \cdots, m+1 \}$ and $\sum\limits^{m+1}_{i =1} T^\ast_i T_i = \mathrm{Id}_\mathcal{H}$. Then we have
\begin{equation*}
\phi(\cdot) = \sum^{m+1}_{i =1} T^\ast_i \phi_i(\cdot) T_i = \sum^{m}_{i =1} T^\ast_i \phi_i(\cdot) T_i + T^\ast_{m+1} \phi_{m+1}(\cdot) T_{m+1}.
\end{equation*}
Let $S$ be the positive square root of $\sum\limits^{m}_{i =1} T^\ast_i T_i$. Since each $T_i$ is invertible, $S$ is invertible and we obtain $S^\ast S  + T^\ast_{m+1} T_{m+1} = \mathrm{Id}_\mathcal{H}$. Define $\psi(\cdot) := \sum\limits^{m}_{i =1} S^{-1\ast} T^\ast_i \phi_i(\cdot) T_i S^{-1}$. Then $\psi \in \ucpapg$ and we have 
\begin{equation*}
\phi(\cdot) = S^\ast \psi(\cdot) S + T^\ast_{m+1} \phi_{m+1}(\cdot) T_{m+1}.
\end{equation*}
Hence by assumption, we get $\psi$ and $\phi_{m+ 1}$ are unitarily equivalent to $\phi$. Now by using the fact that $S$ is the positive square root of $\sum\limits^{m}_{i =1} T^\ast_i T_i$ we obtain that $\psi(\cdot) = \sum\limits^{m}_{i =1} S^{-1\ast} T^\ast_i \phi_i(\cdot) T_i S^{-1}$ is a proper $C^\ast$-convex combination of $\psi$. By using induction hypothesis, we conclude that each $\phi_i$ is unitarily equivalent to $\psi$ and hence each $\phi_i$ is unitarily equivalent to $\phi$. This completes the proof.
\end{proof}

Next, we recall the definition of a $\phi$-invertible operator in \cite{FZ}. Motivated by the results presented in \cite{FZ}, we then prove Theorem \ref{thm; equivalent condition 1}.

\begin{definition} \cite{FZ} \label{def; phi invertible}
Let  $\phi \in \ucpapg$ with the minimal Stinespring representation for $\phi$ be given by $\big (\pi, V, \mathcal{K} \big )$. Then a positive operator $S \in \pi(\mathcal{A})^\prime$ is said to be $\phi$-invertible, if $V^\ast S V$ is invertible in $\mathcal{B}(\mathcal{H})$.
\end{definition}

\begin{theorem} \label{thm; equivalent condition 1}
Let $\mathcal{A}$ be a unital $C^\ast$-algebra and let $\mathcal{H}$ be a Hilbert space. Suppose $\phi$ belongs to  $\ucpapg$. Let $\big (\pi, V, \mathcal{K} \big )$  be the minimal Stinespring representation for $\phi$ and $U_\phi$ be the representation of $G$ on $\mathcal{K}$ as given in Remark \ref{rem; unitary rep of G}. Then $\phi$ is a $C^\ast$-extreme point of $\ucpapg$ if and only if for every $\phi$-invertible positive $S \in \big (\pi(\mathcal{A}) \cup U_\phi(G) \big )^\prime$, there exists 
\begin{enumerate}
\item a family $\big \{ W_g \in \mathcal{B}(\mathcal{K}) \big \}_{g \in G}$ of partial isometries satisfying $W_g \pi(a) = \pi(\tau_g (a)) W_g$, for all $g \in G$ and $a \in \mathcal{A}$, with $\text{Range}(W^\ast_g) = \text{Range}(W^\ast_g W_g) = \overline{\text{Range}(S^{\frac{1}{2}})}$, for all $g \in G$;
\item a family $\big \{ Z_g \in \mathcal{B}(\mathcal{H}) \big \}_{g \in G}$ of invertible operators such that $W_g S^{\frac{1}{2}} V = V Z_g$.
\end{enumerate}
\end{theorem}
\begin{proof}
First, assume that $\phi$ is a $C^\ast$-extreme point of $\ucpapg$ and $S$ is a $\phi$-invertible positive operator in  $\big (\pi(\mathcal{A}) \cup U_\phi(G) \big )^\prime$. Then there exists a real number $\alpha \in (0, 1)$ such that $S_1 := \alpha S$ and $S_2 :=\mathrm{Id}_\mathcal{K} - \alpha S$ are $\phi$-invertible positive operators and moreover, $S_1, S_2 \in \big (\pi(\mathcal{A}) \cup U_\phi(G) \big )^\prime$. Thus for all $g \in G$ and $a \in \mathcal{A}$, we have
\begin{align*}
V^\ast \pi(\tau_g (a)) V = \phi(\tau_g (a)) =  \phi(a) &= V^\ast \pi(a) V \\
&= V^\ast \pi(a) (S_1 + S_2) V \\
&= V^\ast S_1^{\frac{1}{2}}\pi(a) S_1^{\frac{1}{2}} V + V^\ast S_2^{\frac{1}{2}} \pi(a) S_2^{\frac{1}{2}} V. 
\end{align*}
Define $R_i = V^\ast S_i V$ and $T_i = R_i^{\frac{1}{2}}$ for $i \in \{1, 2\}$. By following Definition \ref{def; phi invertible}, we get that $T_i$ is an invertible operator for each $i$. Using this, for each fixed $g \in G$ and $a \in \mathcal{A}$, we see that
\begin{align*}
V^\ast \pi(\tau_g (a)) V &= \phi(\tau_g (a)) \\
&=  \phi(a) \\
&= T^\ast_1 T^{\ast -1}_1 \big ( V^\ast  S^{\frac{1}{2}}_1 \pi(a) S^{\frac{1}{2}}_1 V \big ) T^{-1}_1  T_1 + T^\ast_2 T^{\ast -1}_2 \big ( V^\ast  S^{\frac{1}{2}}_2 \pi(a) S^{\frac{1}{2}}_2 V \big )  T^{-1}_2 T_2  \\
&= T^\ast_1 \big (  T^{\ast -1}_1 V^\ast  S^{\frac{1}{2}}_1 \pi(a) S^{\frac{1}{2}}_1 V T^{-1}_1 \big ) T_1 + T^\ast_2 \big (  T^{\ast -1}_2 V^\ast  S^{\frac{1}{2}}_2 \pi(a) S^{\frac{1}{2}}_2 V T^{-1}_2 \big ) T_2
\end{align*}
Let us define $\phi_1(a) := T^{\ast -1}_1 V^\ast  S^{\frac{1}{2}}_1 \pi(a) S^{\frac{1}{2}}_1 V T^{-1}_1$ and $\phi_2(a) := T^{\ast -1}_2 V^\ast  S^{\frac{1}{2}}_2 \pi(a) S^{\frac{1}{2}}_2 V T^{-1}_2$ for all $a \in \mathcal{A}$. Then it is clear that 
\begin{equation*}
\phi_1(a) = T^{\ast -1}_1 V^\ast  S^{\frac{1}{2}}_1 \pi(a) S^{\frac{1}{2}}_1 V T^{-1}_1 = T^{\ast -1}_1 V^\ast  S^{\frac{1}{2}}_1 \pi(\tau_g (a)) S^{\frac{1}{2}}_1 V T^{-1}_1 =     \phi_1(\tau_g (a)),   
\end{equation*}
for all $g \in G$ and for all $a \in \mathcal{A}$. This shows that $\phi_1 \in \ucpapg$. Similarly, $\phi_2 \in \ucpapg$. Thus for a fixed $g \in G$, we have
\begin{equation} \label{eqn; phi = phi 1 + phi 2}
V^\ast \pi(\tau_g (a)) V = \phi(\tau_g (a)) =  \phi(a) =  T^\ast_1  \phi_1(a) T_1 + T^\ast_2 \phi_2(a) T_2, \; \; \; \text{for all}   \; \; \;  a \in \mathcal{A}.
\end{equation}
Now observe that
\begin{equation*}
T_1^\ast T_1 + T_2^\ast T_2 = R_1 + R_2 = V^\ast ( S_1 + S_2 ) V = \mathrm{Id}_\mathcal{H}.
\end{equation*}
This shows that the $C^\ast$-convex combination given in Equation \eqref{eqn; phi = phi 1 + phi 2} is a proper $C^\ast$-convex combination. Since $\phi$ is a $C^\ast$-extreme point of $\ucpapg$, for a fixed $g \in G$, there exists a unitary $U_{1, g} \in \mathcal{B}(\mathcal{H})$ such that 
\begin{equation*}
 V^\ast \pi(\tau_g (a)) V  =  U_{1, g}^\ast \phi_1(a) U_{1, g} =  (S^{\frac{1}{2}}_1 V T^{-1}_1 U_{1, g})^\ast \pi(a) S^{\frac{1}{2}}_1 V T^{-1}_1 U_{1, g} \; \; \; \text{for all}  \; \;  a \in \mathcal{A}.
\end{equation*}
Let $\mathcal{K}_{1, g}$ be the Hilbert subspace of $\mathcal{K}$, where 
\begin{equation*}
\mathcal{K}_{1, g} := \overline{\text{span}} \left \{  \pi(a) S^{\frac{1}{2}}_1 V T^{-1}_1 U_{1, g}h \; \; : \; \; a \in \mathcal{A}, \; \; h \in \mathcal{H} \right \} \subseteq \mathcal{K}.
\end{equation*}
Then by using the facts that $S_1 = \alpha S \in \big (\pi(\mathcal{A}) \cup U_\phi(G) \big )^\prime$ and the operators $T^{-1}_1$ and $U_{1, g}$ are invertible operators in $\mathcal{B}(\mathcal{H})$, we get
\begin{align*}
\mathcal{K}_{1, g} &= \overline{\text{span}} \left \{  \pi(a) S^{\frac{1}{2}}_1 V T^{-1}_1 U_{1, g}h \; \; : \; \; a \in \mathcal{A}, \; \; h \in \mathcal{H} \right \} \\ 
&= \overline{\text{span}} \left \{  \alpha^{\frac{1}{2}} S^{\frac{1}{2}} \pi(a)  V T^{-1}_1 U_{1, g}h \; \; : \; \; a \in \mathcal{A}, \; \; h \in \mathcal{H} \right \} \\
&= \overline{\text{span}} \left \{  S^{\frac{1}{2}} \pi(a)  V h \; \; : \; \; a \in \mathcal{A}, \; \; h \in \mathcal{H} \right \}  \\
&= \overline{\text{span}} \left \{  S^{\frac{1}{2}} k \; \; : \; \; k \in \mathcal{K} \right \},
\end{align*}
where the last step is achieved by knowing that the minimal Stinespring representation for $\phi$ is given by $\big (\pi, V, \mathcal{K} \big )$. Thus 
\begin{equation} \label{eqn; K 1 g}
\mathcal{K}_{1, g} = \overline{\text{Range}(S^{\frac{1}{2}})}
\end{equation}
Thus for a fixed  $g \in G$, we have the following two minimal Stinespring representations for $\phi$ 
\begin{equation*}
\big (\pi, S^{\frac{1}{2}}_1 V T^{-1}_1 U_{1, g}, \mathcal{K}_{1, g} \big ) \; \; \; \text{and} \; \; \; \big (\pi \circ \tau_g, V, \mathcal{K} \big ).
\end{equation*}
Since the minimal Stinespring representation is unique, by following Proposition \ref{prop;umsd}, we get a unitary $w_g : \mathcal{K}_{1, g} \rightarrow \mathcal{K}$ such that 
\begin{equation*}
w_g (S^{\frac{1}{2}}_1 V T^{-1}_1 U_{1, g}) = V \; \; \;  \text{and}  \; \; \; w_g \pi(a) \big|_{\mathcal{K}_{1, g}} = \pi(\tau_g (a)) w_g.
\end{equation*}
Now extend $w_g$ to the Hilbert space $\mathcal{K}$ by defining it to be 0 on the complement of $\mathcal{K}_{1, g}$, and call this map $W_g \in \mathcal{B}(\mathcal{K})$. Then $W_g$ is a partial isometry in $\mathcal{B}(\mathcal{K})$. Also note that 
\begin{equation} \label{eqn; Wg = V and Wgpi = piWg}
W_g (S^{\frac{1}{2}}_1 V T^{-1}_1 U_{1, g}) = V \; \; \;  \text{and}  \; \; \; W_g \pi(a) = \pi(\tau_g (a)) W_g.
\end{equation}
By following Equation \eqref{eqn; K 1 g}, and the definition of $W_g$, we have $\text{Range}(W^\ast_g) = \text{Range}(W^\ast_g W_g) = \overline{\text{Range}(S^{\frac{1}{2}})}$. It follows from Equation \eqref{eqn; Wg = V and Wgpi = piWg} that $W_g S^{\frac{1}{2}}_1 V = V U_{1, g}^\ast T_1$, where $U_{1, g}^\ast T_1 \in \mathcal{B}(\mathcal{H})$ is invertible. Now we use the fact that $S_1 = \alpha S$ and conclude that $W_g S^{\frac{1}{2}} V = V \frac{U_{1, g}^\ast T_1}{\alpha}$. Let us define $Z_g := \frac{U_{1, g}^\ast T_1}{\alpha}$. Then we get $W_g S^{\frac{1}{2}} V = V Z_g$.  Finally, as the $\phi$-invertible positive operator $S \in \big (\pi(\mathcal{A}) \cup U_\phi(G) \big )^\prime$ and $g \in G$ were arbitrarily chosen, we get the necessary condition. 

Conversely, now assume that for every $\phi$-invertible positive $S \in \big (\pi(\mathcal{A}) \cup U_\phi(G) \big )^\prime$, there exists a family $\big \{ W_g \in \mathcal{B}(\mathcal{K}) \big \}_{g \in G}$ of partial isometries satisfying $W_g \pi(a) = \pi(\tau_g (a)) W_g$, for all $g \in G$ and $a \in \mathcal{A}$, with $\text{Range}(W^\ast_g) = \text{Range}(W^\ast_g W_g) = \overline{\text{Range}(S^{\frac{1}{2}})}$, for all $g \in G$. Also, there exists a family $\big \{ Z_g \in \mathcal{B}(\mathcal{H}) \big \}_{g \in G}$ of invertible operators such that $W_g S^{\frac{1}{2}} V = V Z_g$. Now we make use of Theorem \ref{thm; equivalent condition 3}. Let $\phi(a) = \sum\limits^2_{i=1} T^\ast_i \phi_i(a) T_i$, be a proper $C^\ast$-convex combination of elements of $\ucpapg$. Fix $i \in \{ 1, 2 \}$ and observe that $T^\ast_i \phi_i T_i \in \ucpapg$ and $T^\ast_i \phi_i T_i \leq \phi$. It follows from Lemma \ref{lem;rnd} that there exists a unique $S_i \in \big (\pi(\mathcal{A}) \cup U_\phi(G) \big )^\prime$ such that 
\begin{equation} \label{eqn; for i tistar phi i ti}
T^\ast_i \phi_i(a) T_i = V^\ast S_i \pi(a) V = V^\ast S^\frac{1}{2}_i \pi(a) S^\frac{1}{2}_i V, \; \; \; \; \text{for all} \; \; a \in \mathcal{A}.
\end{equation}
Using the above equation, we get
\begin{equation*}
T^\ast_i T_i = T^\ast_i \phi_i(1_\mathcal{A}) T_i = V^\ast S_i \pi(1_\mathcal{A}) V = V^\ast S_i V.
\end{equation*}
As $T^\ast_i T_i$ is invertible, we get that $V^\ast S_i V$ is also invertible. Thus from the assumption,  there exists a family $\big \{ W_g \in \mathcal{B}(\mathcal{K}) \big \}_{g \in G}$ of  partial isometries  satisfying the property that $W_g \pi(a) = \pi(\tau_g (a)) W_g$ for all $g \in G$ and $a \in \mathcal{A}$, with $\text{Range}(W^\ast_g) = \text{Range}(W^\ast_g W_g) = \overline{\text{Range}(S_i^{\frac{1}{2}})}$ for all $g \in G$. Also there exists and a family $\big \{ Z_g \in \mathcal{B}(\mathcal{H}) \big \}_{g \in G}$ of invertible operators such that $W_g S_i^{\frac{1}{2}} V = V Z_g$. 
Since for all $g \in G$, we have $\text{Range}(W^\ast_g W_g) = \overline{\text{Range}(S_i^{\frac{1}{2}})}$ we obtain $W^\ast_g W_g S^\frac{1}{2}_i = S^\frac{1}{2}_i$. Thus for all $a \in \mathcal{A}$, and $g \in G$ (by using Equation \eqref{eqn; for i tistar phi i ti}), we get
\begin{align*}
\phi_i(a) &= (S^\frac{1}{2}_i V T^{-1}_i)^\ast \pi(a) (S^\frac{1}{2}_i V T^{-1}_i) \\
&= (W^\ast_g W_g S^\frac{1}{2}_i V T^{-1}_i)^\ast  \pi(a)   (S^\frac{1}{2}_i V T^{-1}_i) \\
&= (S^\frac{1}{2}_i V T^{-1}_i)^\ast W^\ast_g W_g  \pi(a)   (S^\frac{1}{2}_i V T^{-1}_i) \\
&= (S^\frac{1}{2}_i V T^{-1}_i)^\ast W^\ast_g \pi(\tau_g (a)) W_g (S^\frac{1}{2}_i V T^{-1}_i) \\
&= (W_g S^\frac{1}{2}_i V T^{-1}_i)^\ast \pi(\tau_g (a)) (W_g S^\frac{1}{2}_i V T^{-1}_i) \\
&= (V Z_g T^{-1}_i)^\ast \pi(\tau_g (a)) (V Z_g  T^{-1}_i) \\
&= (Z_g T^{-1}_i)^\ast V^\ast \pi(\tau_g (a)) V (Z_g  T^{-1}_i) \\
&= (Z_g T^{-1}_i)^\ast \phi(\tau_g (a)) (Z_g  T^{-1}_i) \\
&= (Z_g T^{-1}_i)^\ast \phi(a) (Z_g  T^{-1}_i).
\end{align*}
As $\phi_i(1_\mathcal{A}) = \phi(1_\mathcal{A}) = \mathrm{Id}_\mathcal{H}$, the operator $Z_g  T^{-1}_i$ is an isometry, also, we know that $Z_g$ and $T_i$ are both invertible. Therefore, $Z_g  T^{-1}_i$ is unitary. As $i \in \{ 1, 2  \}$ was fixed arbitrarily, we get that $\phi$ is a $C^\ast$-extreme point of $\ucpapg$. This completes the proof.
\end{proof}

In the following theorem, we use the techniques presented by Zhou in \cite{Z} and give an equivalent condition for an element in $\ucpapg$ to be $C^\ast$-extreme.

\begin{theorem}  \label{thm; equivalent condition 2}
Let $\phi \in \ucpapg$ with the minimal Stinespring representation for $\phi$ be given by $\big (\pi, V, \mathcal{K} \big )$. Then the following statements are equivalent
\begin{enumerate}
\item $\phi$ is a $C^\ast$-extreme point of $\ucpapg$;
\item if $\psi \in \cpapg$ such that $\psi \leq \phi$ and $\psi(1_\mathcal{A})$ is invertible, then there exists an invertible operator $Z \in \mathcal{B}(\mathcal{H})$ such that $\psi(a) = Z^\ast \phi(a) Z$ for all $a \in \mathcal{A}$.
\end{enumerate}
\end{theorem}
\begin{proof}
First we assume that $\phi$ is a $C^\ast$-extreme point of $\ucpapg$. Consider an element $\psi$ in $\cpapg$ such that $\psi \leq \phi$ and $\psi(1_\mathcal{A})$  invertible.  By following Lemma \ref{lem;rnd}, we get a unique $T \in  \big (\pi(\mathcal{A}) \cup U_\phi(G) \big )^\prime$ such that 
\begin{equation*}
\psi(a) = V^\ast T \pi(a) V = (T^{\frac{1}{2}} V)^\ast \pi(a) T^{\frac{1}{2}} V  \end{equation*}
for all $a \in \mathcal{A}$. Now let $\alpha \in (0, 1)$. Then by using the assumptions that $\psi \leq \phi$ also $\psi(1_\mathcal{A})$ is invertible and $\phi$ is unital, we get that 
\begin{equation*}
\alpha \psi(1_\mathcal{A}) = \alpha V^\ast T V \; \; \; \text{and} \; \; \; \mathrm{Id}_\mathcal{H} - \alpha V^\ast T V
\end{equation*}
are invertible positive contractions in $\mathcal{B}(\mathcal{H})$. Then the operators $T_1 := (\alpha V^\ast T V)^{\frac{1}{2}}$ and $T_2 := (\mathrm{Id}_\mathcal{H} - \alpha V^\ast T V)^{\frac{1}{2}}$ are also invertible and we obtain $T^\ast_1 T_1 + T^\ast_2 T_2 = \mathrm{Id}_\mathcal{H}$. Consider the following two maps:
\begin{equation*}
\phi_1(\cdot) := \alpha T^{-1^\ast}_1 \psi(\cdot) T^{-1}_1 \; \; \; \text{and} \; \; \; \phi_2(\cdot) := T^{-1^\ast}_2 (\phi - \alpha \psi)(\cdot) T^{-1}_2.
\end{equation*}
Observe that 
\begin{align*}
\phi_1(1_\mathcal{A}) = \alpha T^{-1^\ast}_1 \psi(1_\mathcal{A}) T^{-1}_1 = T^{-1^\ast}_1 \alpha \psi(1_\mathcal{A}) T^{-1}_1 
= T^{-1^\ast}_1 \alpha V^\ast T V T^{-1}_1
= T^{-1^\ast}_1 T_1 T_1  T^{-1}_1 = \mathrm{Id}_\mathcal{H},
\end{align*}
and similarly $\phi_2(1_\mathcal{A}) = \mathrm{Id}_\mathcal{H}$. Since $\psi$ and $\phi - \psi$ are in $\cpapg$, we get that $\phi_1, \phi_2 \in \ucpapg$. Moreover, we obtain 
\begin{equation*}
\phi(\cdot) = T^{\ast}_1 \phi_1(\cdot) T_1 + T^{\ast}_2 \phi_2(\cdot) T_2.
\end{equation*}
As $\phi$ is a $C^\ast$-extreme points there exists a unitary $U \in \mathcal{B}(\mathcal{H})$ such that $\phi(\cdot) = U^\ast \phi_1(\cdot) U$. Define $V_1 := \alpha^{\frac{1}{2}} T^{\frac{1}{2}} V T^{-1}_1 U \in \mathcal{B}(\mathcal{H}, \mathcal{K})$. Then by using $\psi(\cdot) = V^\ast T \pi(\cdot) V = (T^{\frac{1}{2}} V)^\ast \pi(\cdot) T^{\frac{1}{2}} V$, we get
\begin{equation*}
V^\ast_1 \pi(\cdot) V_1 = \alpha U^\ast T^{-1^\ast}_1 V^\ast T^{\frac{1}{2}^\ast} \pi(\cdot) T^{\frac{1}{2}} V T^{-1}_1 U = \alpha U^\ast T^{-1^\ast}_1 \psi(\cdot) T^{-1}_1 U = U^\ast \phi_1(\cdot) U = \phi(\cdot).
\end{equation*}
Also, $V^\ast_1 V_1 = \phi(1_\mathcal{A}) = \mathrm{Id}_\mathcal{H}$. Let 
\begin{equation*}
\mathcal{K}_1 := \overline{\text{span}} \big \{  \pi(a)V_1h \; \; : \; \; a \in \mathcal{A}, \; \; h \in \mathcal{H} \big \} \subseteq \mathcal{K}.
\end{equation*}
Then $\big (\pi\big|_{\mathcal{K}_1}, V_1, \mathcal{K}_1 \big )$ is also the minimal Stinespring representation for $\phi$. Thus by following Proposition \ref{prop;umsd}, we get a unitary map $U_1 : \mathcal{K}_1 \rightarrow \mathcal{K}$ such that 
\begin{equation*}
U_1V_1 = V \; \; \; \text{and} \; \; \; U_1^\ast \pi(\cdot) U_1 = \pi(\cdot)\big|_{\mathcal{K}_1}.
\end{equation*}
Now consider
\begin{equation*}
V = U_1V_1 = U_1 \alpha^{\frac{1}{2}} T^{\frac{1}{2}} V T^{-1}_1 U = U_1 T^{\frac{1}{2}} V Z^{-1}, \; \; \; \; \text{where} \; \; Z = \frac{1}{\alpha^{\frac{1}{2}}} U^\ast T_1.
\end{equation*}
Then for all $a \in \mathcal{A}$, we get (in the following computations, we use the fact that the domain of $U_1$ contains the range of $V_1$)
\begin{align*}
Z^\ast \phi(a) Z = Z^\ast V^\ast \pi(a) V Z &= Z^\ast V^\ast U_1 U^\ast_1 \pi(a)  V Z \\
&=  Z^\ast V^\ast U_1 \pi(a)\big|_{\mathcal{K}_1} U^\ast_1 V Z \\
&= (U^\ast_1 V Z)^\ast  \pi(a)\big|_{\mathcal{K}_1} U^\ast_1 V Z \\
&= (U^\ast_1 U_1 T^{\frac{1}{2}} V Z^{-1} Z)^\ast  \pi(a)\big|_{\mathcal{K}_1} U^\ast_1 U_1 T^{\frac{1}{2}} V Z^{-1} Z \\
&= (T^{\frac{1}{2}} V)^\ast  \pi(a)\big|_{\mathcal{K}_1} T^{\frac{1}{2}} V \\
&= (T^{\frac{1}{2}} V)^\ast  \pi(a) T^{\frac{1}{2}} V \\
&= V^\ast T^{\frac{1}{2}} \pi(a) T^{\frac{1}{2}} V \\
&= V^\ast T \pi(a) V = \psi(a).
\end{align*}
This proves that under the given assumption $\psi(\cdot) = Z^\ast \phi(\cdot) Z$ for some invertible operator $z \in \mathcal{B}(\mathcal{H})$.

To prove the other implication, we use Theorem \ref{thm; equivalent condition 3}. Let $\phi(\cdot) = \sum\limits^2_{i=1} T^\ast_i \phi_i(\cdot) T_i$, be a $C^\ast$-convex combination of elements of $\ucpapg$. Fix $i \in \{ 1, 2  \}$ and define $\psi(\cdot) := T^\ast_i \phi_i(\cdot) T_i$. Then $\psi(1_\mathcal{A}) = T^\ast_iT_i$ is an invertible operator. Also, $\psi \in \cpapg$ and $\psi \leq \phi$. Then there exists an invertible operator $Z \in \mathcal{B}(\mathcal{H})$ such that $\psi(a) = Z^\ast \phi(a) Z$ for all $a \in \mathcal{A}$. Now since $\phi$ is unital, observe
\begin{align*}
(Z T^{-1}_i)^\ast Z T^{-1}_i = (T^\ast_i)^{-1} Z^\ast Z T^{-1}_i &= (T^\ast_i)^{-1} Z^\ast \phi(1_\mathcal{A}) Z T^{-1}_i  \\
&= (T^\ast_i)^{-1} \psi(1_\mathcal{A}) T^{-1}_i = (T^\ast_i)^{-1} T^\ast_iT_i T^{-1}_i = \mathrm{Id}_\mathcal{H},
\end{align*}
As both $Z$ and $T_i$ are invertible, by using the above equation, we get that $U := Z T^{-1}_i$ is unitary and 
\begin{equation*}
U^\ast \phi(a) U = (T^\ast_i)^{-1} Z^\ast \phi(a)  Z T^{-1}_i = (T^\ast_i)^{-1} \psi(a)  T^{-1}_i = \phi_i(a), \; \; \; \text{for all} \; \; a \in \mathcal{A}.
\end{equation*}
Since $i \in  \{ 1, 2  \}$ was arbtrarily chosen, $\phi$ is a $C^\ast$-extreme point of $\ucpapg$.
\end{proof}

\section{The Krein--Milman theorem} \label{sec; KM Theorem}

In this section, we prove the Krein--Milman theorem for the $C^\ast$-convex set $\ucpapg$. We start by recalling the definition of the $C^\ast$-convex hull of an arbitrary subset of $\cpa$.

\begin{definition}
Let $S \subseteq \cpa$. Then the  $C^\ast$-convex hull of $S$ is given by
\begin{equation*}
C^\ast-con(S) = \left \{  \sum^n_{i =1} T^\ast_i \phi_i(\cdot) T_i \; \; : \; \; \phi_i \in S, \; T_i \in \mathcal{B}(\mathcal{H}), \; \sum^n_{i =1} T^\ast_i T_i = \mathrm{Id}_\mathcal{H}, \; \; n \geq 1 \right \}.
\end{equation*}
\end{definition}
\noindent
The $C^\ast-con(S)$ is the smallest $C^\ast$-convex set containing $S$.

Now we state and prove the Krein--Milman theorem for the $C^\ast$-convex set $\ucpapg$. We use various versions of the Krein--Milman theorem established in  \cite{BBK, BhKu, FM}.

\begin{theorem}
Let $G$ be a finite group, $\mathcal{A}$ be a unital $C^\ast$-algebra, and $\mathcal{H}$ be a separable Hilbert space. Suppose $\tau$ is an action of the finite group $G$ on $\mathcal{A}$ through (unital) $C^\ast$-automorphisms. If one of the following conditions is true:
\begin{enumerate}
\item $\mathcal{H}$ is finite dimensional;
\item $\mathcal{A}$ is separable or type I factor;
\item $\mathcal{A}$ is commutative,
\end{enumerate}
then we have
\begin{equation*}
\ucpapg = \overline{C^\ast-con} \big (C^\ast-ext \big (\ucpapg \big) \big ),
\end{equation*}
where the closure is taken with respect to BW-topology.
\end{theorem}
\begin{proof}
Let $\{ \phi_i \}^n_{i =1}$ be a collection of elements in $C^\ast-ext \big (\ucpapg \big)$. Let $i \in \{1, 2, \cdots, n \}$ with $T_i \in \mathcal{B}(\mathcal{H})$ be such that $\sum\limits^n_{i =1} T^\ast_i T_i = \mathrm{Id}_\mathcal{H}$, where $n \geq 1$. Then for all $g \in G$ and $a \in \mathcal{A}$, we get
\begin{equation*}
\sum^n_{i =1} T^\ast_i \phi_i(\tau_g(a)) T_i =  \sum^n_{i =1} T^\ast_i \phi_i(a) T_i \; \; \; \; \text{and} \; \; \; \sum^n_{i =1} T^\ast_i \phi_i(1_\mathcal{A}) T_i = \sum^n_{i =1} T^\ast_i T_i = \mathrm{Id}_\mathcal{H}.
\end{equation*}
Thus $\sum\limits^n_{i =1} T^\ast_i \phi_i(\cdot) T_i \in \ucpapg$. Since the set $\ucpapg$ is a closed subset of $\ucpa$ with respect to BW-topology (see \cite{BK3}), it follows that  
\begin{equation*}
\overline{C^\ast-con} \big (C^\ast-ext \big (\ucpapg \big) \big ) \subseteq \ucpapg.
\end{equation*}

Now to show the other inclusion consider $\mathcal{A}^G : = \{ a \in \mathcal{A} \; \; : \; \; \tau_g(a) = a, \; \; \text{for all} \; \;  g \in G  \}$. Clearly, $1_\mathcal{A} \in \mathcal{A}^G$. Also, for any fixed $g_o \in G$ and $a \in \mathcal{A}$, one may observe that 
\begin{equation*}
\tau_{g_0} \left ( \frac{\sum\limits_{g \in G} \tau_g(a)}{|G|} \right ) = \frac{\sum\limits_{g_0 g = g^\prime \in G} \tau_{g^\prime}(a)}{|G|} = \frac{\sum\limits_{g \in G} \tau_g(a)}{|G|}.
\end{equation*}
This shows that for any $a \in \mathcal{A}$, we have $\frac{\sum\limits_{g \in G} \tau_g(a)}{|G|} \in \mathcal{A}^G$. If $a, b \in \mathcal{A}^G$, $\lambda \in \mathbb{C}$, then for all $g \in G$, we obatin
\begin{align*}
\tau_g(a + \lambda b) &= \tau_g(a) + \lambda \tau_g(b) = a + \lambda b; \\
\tau_g(ab) &= \tau_g(a) \tau_g(b) = a b; \\
\tau_g(a^\ast) &= \tau_g(a)^\ast = a^\ast.
\end{align*}
Further, as $\tau_g$ is isometric, we find that $\mathcal{A}^G$ is a $C^\ast$-subalgebra of $\mathcal{A}$.

Now we consider the set $\text{UCP} \big(\mathcal{A}^G, \mathcal{B}  (\mathcal{H}) \big)$ of all unital completely positive maps from $\mathcal{A}^G$ to $\mathcal{B}  (\mathcal{H})$, and define a map 
\begin{equation*}
E : \ucpapg \rightarrow \text{UCP} \big(\mathcal{A}^G, \mathcal{B}  (\mathcal{H}) \big)
\end{equation*}
by $E(\phi) := \phi \big |_{\mathcal{A}^G}$. First we prove that $E$ is one-one. Suppose $E(\phi_1) = E(\phi_2)$. Then for any $a \in \mathcal{A}$, 
\begin{equation*}
\phi_1(a) = \phi_1 \left ( \frac{\sum\limits_{g \in G} \tau_g(a)}{|G|} \right ) = \phi_2 \left ( \frac{\sum\limits_{g \in G} \tau_g(a)}{|G|} \right ) = \phi_2(a).
\end{equation*}
Since $a \in \mathcal{A}$ was arbitrarily chosen we get that $\phi_1 = \phi_2$. Now we prove that $E$ is onto. For this consider an element $\phi \in \text{UCP} \big(\mathcal{A}^G, \mathcal{B}  (\mathcal{H}) \big)$. Then define $\tilde{\phi} : \mathcal{A} \rightarrow \mathcal{B}  (\mathcal{H})$ by 
\begin{equation*}
\tilde{\phi}(a) := \phi \left (  \frac{\sum\limits_{g \in G} \tau_g(a)}{|G|} \right ), \; \; \; \; \text{for all} \; \; a \in \mathcal{A}.
\end{equation*}
Then clearly $\tilde{\phi}(a) = \phi(a)$ for all $a \in \mathcal{A}^G$. This implies that  $E(\tilde{\phi}) = \phi$. Thus $E$ is onto. Also, we get 
\begin{equation*}
 \left ( \phi = \sum^n_{i =1} T^\ast_i \phi_i T_i \right ) \Big |_{\mathcal{A}^G} = \sum^n_{i =1} T^\ast_i \phi_i \big |_{\mathcal{A}^G}  T_i, \; \; \; \text{that is,} \; \; \; E \left (\phi = \sum^n_{i =1} T^\ast_i \phi_i T_i \right )  = \sum^n_{i =1} T^\ast_i E( \phi_i) T_i,
\end{equation*}
where each $\phi_i \in \ucpapg$ and $T_i \in \mathcal{B}(\mathcal{H})$  with $\sum\limits^n_{i =1} T^\ast_i T_i = \mathrm{Id}_\mathcal{H}$. This shows that the map $E$ preserves the $C^\ast$-convex combination. Moreover, the map $E$ has the property that if $\phi_\alpha \rightarrow \phi$ in the BW-topology, then $E(\phi_\alpha) \rightarrow E(\phi)$ in BW-topology as well. Since the domain of $E$ is compact, it follows that the inverse map $E^{-1}$ is also continuous.

Now suppose $\psi \in C^\ast-ext \big (\text{UCP} \big(\mathcal{A}^G, \mathcal{B}  (\mathcal{H}) \big) \big)$. Then we show that $E^{-1}(\psi)$ belongs to $C^\ast-ext \big (\ucpapg \big)$. Suppose $E^{-1}(\psi) = \sum\limits^n_{i =1} T^\ast_i \phi_i T_i$, where $\phi_i \in \ucpapg$, \; $T_i \in \mathcal{B}(\mathcal{H})$, and $T^{-1}_i$ exists for all $i \in \{ 1, 2, \cdots, n \}$ with $\sum\limits^n_{i =1} T^\ast_i T_i = \mathrm{Id}_\mathcal{H}$. Then we have 
\begin{equation*}
\psi = E \left ( E^{-1}(\psi) \right ) = E \left ( \sum\limits^n_{i =1} T^\ast_i \phi_i T_i \right ) = \sum\limits^n_{i =1} T^\ast_i E(\phi_i) T_i.
\end{equation*}
Since $\psi \in C^\ast-ext \big (\text{UCP} \big(\mathcal{A}^G, \mathcal{B}  (\mathcal{H}) \big) \big)$, there exists a family of unitary operators $\big \{ U_i \in  \mathcal{B}(\mathcal{H}) \big \}^n_{i =1}$ such that
\begin{equation*}
E(\phi_i)(\cdot) = U^\ast_i \psi(\cdot) U_i, \; \; \; \text{for all} \; \; 1 \leq i \leq n.
\end{equation*}
This implies
\begin{equation*}
\phi_i(\cdot) = E^{-1}(E(\phi_i))(\cdot) = E^{-1}(U^\ast_i \psi(\cdot) U_i) = U^\ast_i E^{-1}(\psi)(\cdot) U_i, \; \; \; \text{for all} \; \; 1 \leq i \leq n.
\end{equation*}
This shows that $E^{-1}(\psi)$ belongs to $C^\ast-ext \big (\ucpapg \big)$.

We now complete the proof of the result. Let $\phi \in \ucpapg$, and correspondingly consider $E(\phi) \in \text{UCP} \big(\mathcal{A}^G, \mathcal{B}  (\mathcal{H}) \big)$. If we assume that the $C^\ast$-algebra $\mathcal{A}$ and the Hilbert space $\mathcal{H}$ satisfy one of the conditions from (1), (2) or (3) (given in the hypothesis), then the $C^\ast$-algebra $\mathcal{A}^G$ also satisfies the same conditions. By following the results presented in \cite{BBK, BhKu, FM}, we know that there exists a net $\{ \psi_\alpha \}_{\alpha} \subseteq C^\ast-con \big (C^\ast-ext \big (\text{UCP} \big(\mathcal{A}^G, \mathcal{B}  (\mathcal{H}) \big) \big) \big )$ such that 
\begin{equation} \label{eqn; psi alpha goes to e of phi}
\psi_\alpha \rightarrow E(\phi), \; \; \; \; \text{in BW-topology}.
\end{equation}
Write $\psi_\alpha = \sum\limits^{n_\alpha}_{j =1} T^\ast_{\alpha, j} \psi_{\alpha, j} T_{\alpha, j}$, where $\psi_{\alpha, j} \in C^\ast-ext \big (\text{UCP} \big(\mathcal{A}^G, \mathcal{B}  (\mathcal{H}) \big) \big)$, \;  $T_{\alpha, j} \in \mathcal{B}(\mathcal{H})$ and $n_\alpha \in \mathbb{N}$. Then for each $\alpha$ and $j \in \{1, 2, \cdots, n_\alpha \}$, we get  
\begin{equation*}
E^{-1}(\psi_{\alpha, j}) \in C^\ast-ext \big (\ucpapg \big ) \big).
\end{equation*}
Consequently, we get a net $\big \{ E^{-1}(\psi_\alpha) \big \}_{\alpha} \subseteq C^\ast-con \big (C^\ast-ext \big (\ucpapg \big) \big)$. By following Equation \eqref{eqn; psi alpha goes to e of phi}, and knowing the fact that $E^{-1}$ is contiunous with respect to BW-topology, we get that
\begin{equation*}
E^{-1}(\psi_\alpha) \rightarrow \phi \; \; \; \; \text{in BW-topology}.
\end{equation*}
Since $\phi \in \ucpapg$ was arbitrarily chosen, we get that 
\begin{equation*}
\ucpapg \subseteq \overline{C^\ast-con} \big (C^\ast-ext \big (\ucpapg \big) \big ).
\end{equation*}
This proves the result.
\end{proof}

%%%%%%%%%%%%%%%%%%%%%%%%%%%%%%%%%%%%%%%%%%%%%%%%%%%%%%%%%%%%%%%%%%%%%%%%%%%%%%%%%%%%%%%%%

\subsection*{Acknowledgements}
The author sincerely thanks the Indian Institute of Science Education and Research (IISER) Mohali, India for the financial support received through an Institute Postdoctoral Fellowship. 

\subsection*{Availability of data and material} Not applicable.

%\subsection*{Competing interests} The authors declare that there are no conflicts of interest.
%\subsection*{Declaration}
%The authors declare that there are no conflicts of interest.

\bibliographystyle{plain}
%\bibliography{mybib}   

\begin{thebibliography}{99}
\bibitem{AS}  Anand O. R. and Sumesh K. {\em Generalized {$C^*$}-convexity in completely positive maps} J. Math. Anal. and Appl. 551 (2025), Issue No. 2, Part 2, 129700.

\bibitem{Arveson1} Arveson, W. B. {\em Subalgebras of \(C^*\)-algebras}. Acta Math. 123(1969), 141-224.

\bibitem{BBK}  Banerjee, T., Bhat, B. V. R. and Kumar M.   {\em {$C^*$}-extreme points of positive operator valued measures and unital completely positive maps} Commun. Math. Phys, 388 (2021), Issue No. 4, 1235-1280.

\bibitem{BDMS}  Bhat, B. V. R., Devendra R., Mallick N. and Sumesh K.  {\em {$C^*$}-extreme points of entanglement breaking maps} Rev. Math. Phys, 35 (2023), Issue No. 3, Paper No: 2350005, 17.

\bibitem{BH}  Balasubramanian, S. and Hotwani, N.  {\em {$C^*$}-extreme entanglement breaking maps on operator systems} Linear Algebra Appl., 685 (2024), 182-213.


\bibitem{BhKu}  Bhat, B. V. R. and Kumar M.   {\em {$C^*$}-extreme maps and nests} J. Funct. Anal., 282  (2022), Issue No. 8, Paper No: 109397, 40.

%\bibitem{BK1}  Bhattacharya, A. and Kulkarni, C. J. {\em Barycentric decompositions in the space of weak expectations} Adv. Oper. Theory. 8 (2023), Issue No. 4, paper no. 59.

%\bibitem{BK2}  Bhattacharya, A. and Kulkarni, C. J. {\em Generalized orthogonal measures on the space of unital completely positive maps.} Forum Mathematicum 36 (2024), no. 6, 1483-1497.

\bibitem{BK3}  Bhattacharya, A. and Kulkarni, C. J. {\em Ergodic decomposition in the space of unital completely positive maps} Infi. Dimen. Anal. Quantum Probab. and Relat. Top. 28 (2025),  No. 3,  Paper No. 2450012. 


\bibitem{OB1}  Bratteli, O. and Robinson, D. W. {\em Operator algebras and quantum statistical mechanics I. \(C^*\)- and \(W^*\)-algebras, symmetry groups, decomposition of states.} 2nd ed. (1987) Texts and Monographs in Physics.

\bibitem{DK} Davidson K. R. and  Kennedy M. {\em Noncommutative Choquet theory} arxiv: 1905.08436 v3.

\bibitem{EW}  Effros, E. G. and Winkler, S. {\em Matrix convexity: operator analogues of the bipolar and {H}ahn-{B}anach theorems} J. Funct. Anal. 144 (1997), Issue No. 1, 117--152.

%\bibitem{Exel-1994-Circle-action-Cst-alg}
	%Exel, R.
	%\newblock {\em Circle actions on {$C^*$}-algebras, partial automorphisms, and a
	%generalized {P}imsner-{V}oiculescu exact sequence.}
	%\newblock { J. Funct. Anal.}, 122(2):361--401, 1994.
    
%\bibitem{E} Exel, R. {\em Partial Dynamical Systems, Fell Bundles and Applications}, Amer. Math. Soc. Mathematical Surveys and Monographs 224.

\bibitem{FM93}  Farenick, D. R. and Morenz, P. B. {\em {$C^*$}-extreme points of some compact {$C^*$}-convex sets} Proc. Amer. Math. Soc. 118 (1993), Issue No. 3, 765--775.

\bibitem{FM}  Farenick, D. R. and Morenz, P. B. {\em {$C^*$}-extreme points in the generalised state spaces of a {$C^*$}-algebra} Trans. Amer. Math. Soc. 349 (1997), Issue No. 5, 1725--1748.

\bibitem{FZ}  Farenick, D. R. and Zhou, H. {\em The Structure of {$C^*$}-Extreme Points in Spaces of Completely Positive Linear Maps on {$C^*$}-Algebras} Proc. Amer. Math. Soc. 126 (1998), Issue No. 5, 1467--1477.

\bibitem{F}  Fujimoto I. {\em CP duality for {$C^*$}- and {$W^*$}-algebras} J. Operator Theory 30 (1993), 201-215.

\bibitem{G} Gregg, M. C. {\em On {$C^*$}-extreme maps and $\ast$-homomorphisms of a commutative {$C^*$}-algebra} Integral Equations Operator Theory, 63 (2009), Issue No. 3, 337--349.


\bibitem{HK} Hossain, Md. A. and Kulkarni, C. J. {\em Extreme points of Unital Completely Positive maps invariant under partial action} preprint, arXiv:2507.20797 [math.OA], 16p.

\bibitem{L}  Loebl, R. I. {\em A remark on unitary orbits} Bull. Instit. Math. Acad. Sinica 7 (1979), 401--407.

\bibitem{LP}  Loebl, R. I. and Paulsen, V. I. {\em Some remarks on {$C\sp{\ast} $}-convexity} Linear Algebra Appl. 35 (1981), 63--78.

%\bibitem{McClanahan-1995-K-theory-Par-act}
	 %McClanahan, K.
	%\newblock { \em $K$}-theory for partial crossed products by discrete groups.
	%\newblock {J. Funct. Anal.}, 130(1):77--117, 1995.
	


%\bibitem{Paschke} Paschke, W., {\em Inner Product Modules over \(B^*\)-algebras}, Trans. Amer. Math. Soc. 182(1973), 443-468
\bibitem{O}  Ostapenko, V. V. {\em Matrix convexity} Ukrain. Mat. Zh. 47 (1995), Issue No. 1, 64-69.


\bibitem{Paulsen} Paulsen, V. {\em Completely Bounded Maps and Operator Algebras} Cambridge Studies in Advanced Mathematics, 78.

%\bibitem{Phlp} Phelps, Robert R. {\em Lectures on Choquet's theorem.} Second edition. Lecture Notes in Mathematics, 1757. Springer-Verlag, Berlin, 2001. viii+124 pp.

\bibitem{W}  Wu, W. {\em {$C^{\ast} $}-extreme points in  {$W^{\ast} $}-algebras} Acta Math. Sinica (Chinese Ser.) 45 (2002), Issue No. 3, 455--460.

\bibitem{Z}  Zhuo, H. {\em {$C^{\ast} $}-extreme points in  spaces of completely positive maps} ProQuest LLC, Ann Arobor, MI, 1998. Thesis (Ph.D.)-The University of Regina, Canada.



\end{thebibliography}

\end{document}